\documentclass[a4paper,12pt]{amsart}
\usepackage{amsmath,amsthm,amssymb}
\usepackage{hyperref,cleveref}
\usepackage{graphicx}

\allowdisplaybreaks[1]

\newtheorem{thm}{Theorem}
\newtheorem{lemma}{Lemma}
\newtheorem{cor}[thm]{Corollary}

\begin{document}

\title[Functions with no unbounded Fatou components]{Functions with no unbounded fatou components }
\author[R. Kaur]{Ramanpreet Kaur}
\address{Department of Mathematics, University of Delhi,
Delhi--110 007, India}

\email{preetmaan444@gmail.com}

\begin{abstract}
For a transcendental entire function, a partial affirmative answer to Baker's question on the boundedness of its Fatou components is given. In addition, we have addressed Wang's question on Fej\'er gaps. Certain results about functions with  Fabry gaps and of infinite order have also been generalized.
\end{abstract}
\maketitle

\section{Introduction and Preliminaries}
\label{sec1}
This article investigates the question raised by Baker for transcendental entire functions, which says, whether every Fatou component of a transcendental entire function is bounded when the function is of sufficiently small growth. 
Let us recall some basic definitions that we will be using throughout this article.

Let $f$ be a transcendental entire function and let $f^n$ denote the $n\text{-th}$ iterate of $f$. 
The Fatou set $F(f)$ is defined to be the set of all $z\in \mathbb{C}$ such that $\{f^n\}_{n\in\mathbb{N}}$ forms a normal family in some neighbourhood of $z$. 
The complement $J(f)$ of $F(f)$ is called as the Julia set of $f$. 
The basic properties of these sets can be found in \cite{beardon,berg,hua}.

The order of growth $\rho$, lower order of growth $\lambda$, and type $\sigma$ of a transcendental entire function $f$ are defined as follows:
\begin{align*}
\rho&= \limsup_{r\to\infty}\frac{\log \log M(r,f)}{\log r}, \\
\lambda&= \liminf_{r\to\infty}\frac{\log \log M(r,f)}{\log r}, \text{ and}\\
\sigma&=\limsup_{r\to\infty}\frac{\log M(r,f)}{r^{\rho}},
\end{align*}
 where  $M(r,f)=\max\{|f(z)|: |z|=r\}$.
 
The growth of $f$ is said to be of minimal type if $\sigma=0$, mean type if $\sigma\in(0,\infty)$, and maximal type if $\sigma=\infty$.
Also, the minimum modulus of $f$ on $|z|=r$ is denoted by $L(r,f)$, i.e.,
\[L(r,f)=\min\{|f(z)|: |z|=r\}.\]

 For each $r\geq 0$, the maximum term $\mu(r,f)$ is defined as
\[ \mu(r,f)= \max\{|a_k|r^k: k=0,1,2,...\},\]
 where $a_k$'s are the coefficients of $f$ in its power series expansion around zero.
 For $r>0$, the central index $\nu(r,f)$ is defined as the largest $k$ for which $|a_k|r^k=\mu(r,f).$

In \cite{baker1}, Baker observed that for the function
\[f(z)=\frac{\sin z^{\frac{1}{2}}}{z^{\frac{1}{2}}}+z+a,\] where $a$ is sufficiently large, $F(f)$  contains a segment $[x_0,\infty)$ of the positive real axis. This function is of order $\frac{1}{2}$, mean type. Hence, the sufficiently small growth condition appears to be of order $\frac{1}{2}$, minimal type at most.

Along with Baker, several researchers gave an affirmative answer to this question by assuming certain regularity conditions on the function $f$. We list below some of these results.

Baker proved that if $f$ satisfies the condition that
\[\log M(r,f)=O((\log r)^p), \text{ for some }  p\in(1,3),\] 
 then every component of $F(f)$ is bounded.

Stallard improved the growth condition to
\[\log \log M(r,f)<\frac{(\log r)^{\frac{1}{2}}}{(\log \log r)^c},\text{ for some } c\in(0,1)\]
for sufficiently large $r$ \cite{stallard}.  
In the same paper, it is also proved that if $\rho<\frac{1}{2}$ and 
\[\frac{\log \log M(2r,f)}{\log M(r,f)}\to c\not= \infty \text{ as }r\to\infty,\] then every component of $F(f)$ is bounded.

In \cite{zheng}, Zheng proved that there are no unbounded periodic Fatou components if the growth of $f$ is at most of order $\frac{1}{2}$, minimal type. Anderson and Hinkannen used the notion of self sustaining spread to prove some results in this direction \cite{anderson}.  Rippon and Stallard also gave some sufficient conditions which imply that $F(f)$ has no unbounded Fatou component \cite{rippon}. They also improved some results of Hinkannen, Wang \cite{Wang1} and of several other authors. In addition, they proved that Eremenko's conjecture holds for functions with no unbounded Fatou components.  Thus, Baker's question is still open for functions with wandering domains.
\section{Main Result}
In this article, we prove the following:
\begin{thm}\label{THM1}
Let $f$ be a transcendental entire function. If for given $\epsilon>0$, the condition
\[\log L(r,f)> (1-\epsilon)\log  M(r,f)\] is satisfied for every $r$ outside a set of logarithmic density zero, then $F(f)$ has no unbounded components.
\end{thm}

The above theorem generalizes \cite[Theorem 1]{Singh} and \cite[Theorem 1]{Wang1}. 
In \cite{Singh}, Singh considered $f$ as a composition of two transcendental entire functions of positive order with certain conditions on them; meanwhile Wang's paper \cite{Wang1} deals with functions of finite order possessing Fabry gaps with positive lower order.  
Note that functions of finite order with Fabry gaps certainly satisfy the hypothesis of \Cref{THM1} \cite{fuchs}.
In the same paper, Wang also proposed the question whether a function with Fej\'er gaps has no unbounded Fatou components. 
\Cref{THM1} gives an affirmative answer to this question also, as a function with Fej\'er gaps satisfies the condition of \Cref{THM1} \cite{murai}.    

In the proof of \Cref{THM1}, we need two real sequences $\{R_n\}$ and $\{S_n\}$ satisfying the following properties:
\begin{enumerate}
\item For any $\alpha>1$, we have $R_{n+1}=M(R_n^{\frac{1}{2\alpha}},f)$, $S_{n+1}=M(S_n,f)$, and
\item $S_n\leq R_n^{\frac{1}{2\alpha}}$, for every sufficiently large $n$.
\end{enumerate}

Before considering the construction of the sequences in the most general case of $f$ being any transcendental entire function, we shall illustrate a construction for a special case as given below.

Suppose that $f$ is of positive lower order and having finite order. 
As $\frac{\rho}{\lambda}>0$, there exists $n_{\lambda,\rho}\in\mathbb{N}$ such that $n_{\lambda,\rho}>\frac{\rho}{\lambda}$. 
Now, consider
\begin{align*}
\liminf_{r\to\infty}\frac{\log \log M(r^{\frac{1}{2\alpha}},f)^{\frac{1}{16 \alpha^4 n_{\lambda,\rho}}}}{\log r^{2\alpha}}&=\liminf_{r\to\infty}\left(\frac{\log \frac{1}{16\alpha^4n_{\lambda,\rho}}}{\log r^{2\alpha}}+\frac{\log\log M(r^{\frac{1}{2\alpha}},f)}{\log r^{2\alpha}}\right) \\
&=\liminf_{r\to\infty}\frac{\log\log M(r^{\frac{1}{2\alpha}},f)}{\log r^{2\alpha}} \\
&= \frac{\lambda}{4\alpha^2} \text{\quad(by the definition of lower order).}
\end{align*}
Hence, there exists $r_3>0$ such that
\[\frac{\log \log M(r^{\frac{1}{2\alpha}},f)^{\frac{1}{16 \alpha^4 n_{\lambda,\rho}}}}{\log r^{2\alpha}}\geq \frac{\lambda}{8\alpha^2} \text{\quad for every } r\geq r_3.\]

On similar lines, using the definition of order of $f$,
there exists $r_4>0$ such that
 \[ \log \log M(S_n,f)^{\frac{1}{2\alpha}}\leq 2\rho\log S_n \text{ for every }r\geq r_4.\]

Now, take $R_1,S_1> \max\{r_0,r_1,r_3,r_4\}$ such that $R_1^{\frac{1}{16\alpha^4n_{\lambda,\rho}}}\geq S_1^{\frac{1}{2\alpha}}$, where $r_0,r_1$ are chosen from the steps 1 and 2 of the proof of \Cref{THM1} (to follow). 
Further, for any $\alpha>1$, set $R_{n+1}=M(R_n^{\frac{1}{2\alpha}},f)$ and $S_{n+1}=M(S_n,f)$. 
To prove $(2)$, it is sufficient to show that $S_n^{\frac{1}{2\alpha}}\leq R_n^{\frac{1}{16\alpha^4n_{\lambda,\rho}}}$ for every $n\in\mathbb{N}$, which we shall prove by induction on $n$.\\
Consider
\begin{align*}
\frac{\log \log M(R_n^{\frac{1}{2\alpha}},f)^{\frac{1}{16 \alpha^4 n_{\lambda,\rho}}}}{\log \log M(S_n,f)^{\frac{1}{2\alpha}}}&\geq \frac{\frac{\lambda}{8\alpha^2}\log R_n^{2\alpha}}{2\rho\log S_n} \\
&\geq \frac{\lambda \log R_n^{2\alpha}}{16\alpha^2 \rho\log S_n^{\alpha^2}}\\
&\geq \frac{\log R_n}{16\alpha^4n_{\lambda,\rho}\log S_n^{\frac{1}{2\alpha}}}\\
&=\frac{\log R_n^{\frac{1}{16\alpha^4n_{\lambda,\rho}}}}{\log S_n^{\frac{1}{2\alpha}}}\geq 1.
 \end{align*}

Consequently, \[S_{n+1}^{\frac{1}{2\alpha}}=M(S_n,f)^{\frac{1}{2\alpha}}\leq M(R_n^{\frac{1}{2\alpha}},f)^{\frac{1}{16 \alpha^4 n_{\lambda,\rho}}}= R_{n+1}^{\frac{1}{16\alpha^4n_{\lambda,\rho}}},\] which completes the induction process.

%We shall now illustrate a construction of sequences $\{R_n\}, \{S_n\}$ satisfying $(1)$ and $(2)$ for any transcendental entire function. 
General Case: Here, lower order of $f$ can be zero or order of $f$ can be infinity, and hence we can not apply the method as done above.
To this end, we first make the following observations.

Using \cite[Lemma 2.2.7]{langley}, we may choose a constant $K=\log \mu(s,f)\geq 2$ for sufficiently large $s\geq 1 $ such that 
\[\log \mu(r,f)\leq \nu(r,f)\log r +K \text{ for  sufficiently large }r.\]
Replacing $r$ with $2r$, we get 
\begin{equation}\label{1}
\log \mu(2r,f)\leq \nu(2r,f)\log 2r +K \text{ for  sufficiently large }r.
\end{equation}
Now, for $r>0$, by \cite[Lemma 2.2.2]{langley}, we have $M(r,f)\leq 2\mu(2r,f)$. Taking $\log$ on both sides, we get
 \[\log M(r,f)\leq \log 2+ \log \mu(2r,f) \text{ for  sufficiently large }r.\]
This further implies  that 
\[\log 2 M(r,f)\leq 2 \log 2 +  \log \mu(2r,f) \text{ for  sufficiently large }r.\]
Hence,
\begin{align*}
\log 2 M\left(r,f\right) & \leq 2 \log 2+ \nu(2r,f)\log 2r +K\\
& \leq K' + \nu(2r,f)\log 2r \\
& = \log K'' + \nu(2r,f)\log 2r \\
& = \log( K'' 2r^{\nu(2r,f)} )
\end{align*}
 for sufficiently large $r$. Therefore, there exists $s_0>0$ such that 
 \[\log 2 M\left(r,f\right)\leq \log( K'' (2r)^{\nu(2r,f)} ) \text{ for every } r\geq s_0.\]
 
 Again, using \cite[Lemma 2.2.7]{langley}, we have 
\[ \log \mu(r,f)+\nu(r,f)\log r\leq \log \mu(r^2,f), \text{ for sufficiently large } r.\]
  Using \cite[Lemma 2.2.2]{langley}, we obtain 
  \begin{equation}\label{2}
 \log \mu(r,f)+\nu(r,f)\log r\leq \log M(r^2,f) \text{ for sufficiently large } r.
  \end{equation}
 This gives us the existence of $s_1>0$ such that 
 \[\log \mu(r,f)+\nu(r,f)\log r\leq \log M(r^2,f) \text{ for every } r\geq s_1.\]
Choose $R_1, S_1$ such that $R_1^{\frac{1}{4\alpha}}\geq 2S_1\geq \max\{s_0,s_1,r_1,r_2\}$ and $\mu(R_1^{\frac{1}{4\alpha}},f)\geq K''$. For  
$n\in\mathbb{N}$, define $R_{n+1}=M(R_n^{\frac{1}{4\alpha}},f),$ 
 and $ S_{n+1}=M(S_n,f).$ 
\begin{lemma}\label{lem1}
There exists a sequence $\{k_n\}$ in $(0,1)$ such that 
\[ a_n:=\frac{\nu(R_n^{\frac{1}{4\alpha}},f)}{\nu((8S_n)^{2k_n},f)(\log M(4S_n^2,f))^2}\] satisfies $k_n\leq a_n$ for every $n\in\mathbb{N}$, and $(8S_n)^{2k_n}$ approaches to some finite number as $n$ tends to infinity.
\end{lemma}
\begin{proof}
Firstly, we choose a sequence $\{l_n\}$ of real numbers such that $(8S_n)^{2l_n}$ approaches to some $b$ as $n$ tends to infinity. On using right continuity of $\nu(r,f)$ at $b$, we have the existence of $\delta>0$ such that 
$\nu((8S_n)^{2l_n},f)= \nu(b,f)$ for every $n$ satisfying $b<(8S_n)^{2l_n}< b+\delta$.

Now, consider $\frac{\nu(R_n^{\frac{1}{4\alpha}},f)}{(\log M(4S_n^2,f))^2}= b_n$ (say). The above observations gives us that
\[a_n=\frac{b_n}{\nu(b,f)} \text{ for sufficiently large }n.\]
Now, for every $n\in\mathbb{N}$, define $k_n= \min\{a_n,l_n\}$. From the definition, it is clear that $\{k_n\}$ satisfies the required properties, i.e., $k_n\leq a_n $ and $(8S_n)^{2k_n}$ approaches to some finite number, say $a$ as $n$ tends to infinity.
\end{proof}

Now, using the above observations and \Cref{lem1}, we shall prove that 
 \[\nu(2S_n,f)\leq \frac{a_n\nu((8S_n)^{2 k_n},f)(\log M(4S_n^2,f))^2}{4\alpha} \text{ for sufficiently large  } n.\]
 For this, consider
 \begin{align}
 \frac{\nu(2S_n,f)}{\nu((8 S_n)^{2 k_n},f)(\log M(4S_n^2,f))^2}& \leq \frac{K(\log M(4S_n^2,f)-\log \mu(2S_n,f)) \log (8S_n)^{2k_n}}{\log 2S_n \log M\left(\frac{(8S_n)^{2k_n}}{2},f\right)(\log M(4S_n^2,f))^2}\label{3}\\
 & = \frac{K \log M(4S_n^2,f)\log(8S_n)^{2k_n}}{\log 2S_n \log M\left(\frac{8S_n)^{2k_n}}{2},f\right)(\log M(4S_n^2,f))^2}\notag\\
 &\quad -\frac{K\log \mu(2S_n,f)}{\log 2S_n \log M\left(\frac{(8S_n)^{2k_n}}{2},f\right) (\log M(4S_n^2,f))^2}\notag\\
 & \leq \frac{K \log M(4S_n^2,f)\log (2S_n)^{6k_n} }{\log 2S_n\log M\left(\frac{(8S_n)^{2k_n}}{2},f\right)(\log M(4S_n^2,f))^2}\notag\\
 &= \frac{6 K k_n \log 2S_n }{\log 2S_n \log M\left(\frac{(8S_n)^{2k_n}}{2},f\right)\log M(4S_n^2,f)}\notag\\
 & = \frac{6 K a_n}{m_0\log M(4S_n^2,f)}\label{4} \\
 & \leq \frac{a_n}{4\alpha},\label{5} 
   \end{align}
    for sufficiently large $n$.

 This means that there exists $n_1\in\mathbb{N}$ such that 
 \[\nu(2S_n,f)\leq \frac{a_n\nu((8S_n)^{2 k_n},f)\log (8S_n)^{2k_n}} {4\alpha} \text{ for every }n\geq n_1.\]
  Justification for the deduction on each line of the above multiline equation:
  \begin{itemize}
  \item  On the first line (\Cref{3}), we have used 
\Cref{1}, \Cref{2} and  the following inequality:
for $\beta>1$, we have
$x-\beta\geq \frac{x}{\beta}, \text{ for sufficiently large }x.$
\item For the fifth inequality (\Cref{4}), using \Cref{lem1}, we get $k_n\leq a_n$ and $\log M\left(\frac{(8S_n)^{2k_n}}{2},f\right)$ approaches to $\log M\left(\frac{a}{2},f\right)= m_0$ (say).\\
\item For the sixth inequality (\Cref{5}), use the fact that
$\frac{1}{\log M(4S_n^2,f)}$ approaches to zero as $n$ tends to $\infty.$
\end{itemize}
Now, assume that $R_{n_1}^{\frac{1}{4\alpha}}\geq 2S_{n_1}$. Then, for $n> n_1$, consider
  \begin{align*}
\frac{\log M(R_{n_1}^{\frac{1}{2\alpha}},f)^{\frac{1}{4\alpha}}}{\log 2 M(S_{n_1},f)} & \geq \frac{\log \mu( R_{n_1}^\frac{1}{4\alpha},f)^{\frac{1}{4\alpha}}R_{n_1}^{\frac{ \nu(R_{n_1}^{\frac{1}{4\alpha}},f)}{16\alpha^2}}}{\log \left(K'' (2S_{n_1})^{\nu(2S_{n_1},f)}\right)}\\
&\geq \frac{\log \mu( R_{n_1}^\frac{1}{4\alpha},f)^{\frac{1}{4\alpha}}R_{n_1}^{\frac{ \nu(R_{n_1}^{\frac{1}{4\alpha}},f)}{16\alpha^2}}}{\log \left(K'' (2S_{n_1}^{a_{n_1}})^{\frac{a_n\nu((8S_{n_1})^{2k_{n_1}},f)\log (8S_n)^{2k_n}}{4\alpha}}\right)}\geq 1.\\
\end{align*}  
   This gives us that $R_{n_1+1}^{\frac{1}{4\alpha}}\geq 2S_{n_1+1}$. On applying the same process inductively, we get 
  \[R_n^{\frac{1}{4\alpha}}\geq 2S_n \text{ for every }n\geq n_1.\]

We now prove the following result which will also be used in the proof of \Cref{THM1}.

\begin{lemma}\label{lemma1}
Let $f$ be a transcendental entire function and let $m>1$.
Then,
\[M(\log r^{\frac{1}{2m}r^{\frac{1}{2m}}},f)\geq \left(\log M(r^{\frac{1}{2m}},f)^{\frac{1}{2m}M(r^{\frac{1}{2m}},f)^{\frac{1}{2m}}}\right)^{m}\]
for sufficiently large $r$.
\end{lemma}
\begin{proof}
Consider, 
\begin{align*}
\log M(r^{\frac{1}{2m}},f)^{\frac{1}{2m}M(r^{\frac{1}{2m}},f)^{\frac{1}{2m}}}&= M(r^{\frac{1}{2m}},f)^{\frac{1}{2m}}\log M(r^{\frac{1}{2m}},f)^{\frac{1}{2m}}\\
&\leq M(r^{\frac{1}{2m}},f)^{\frac{1}{2m}} M(r^{\frac{1}{2m}},f)^{\frac{1}{2m}}\\
&= M(r^{\frac{1}{2m}},f)^{\frac{1}{m}}.
\end{align*}
As $\log r^{\frac{1}{2m}r^{\frac{1}{2m}}}\geq r^{\frac{1}{2m}}$ for sufficiently large $r,$ we have $M(\log r^{\frac{1}{2m}r^{\frac{1}{2m}}},f)\geq M(r^{\frac{1}{2m}},f)$ for sufficiently large $r$. This further implies that \[\frac{\log M(r^{\frac{1}{2m}},f)^{\frac{1}{2m}M(r^{\frac{1}{2m}},f)^{\frac{1}{2m}}}}{M(\log r^{\frac{1}{2m}r^{\frac{1}{2m}}},f)^{\frac{1}{m}}}\leq 1\]
for sufficiently large $r$. As a result, we obtain
\[\left(\log M(r^{\frac{1}{2m}},f)^{\frac{1}{2m}M(r^{\frac{1}{2m}},f)^{\frac{1}{2m}}}\right)^{m}\leq M(\log r^{\frac{1}{2m}r^{\frac{1}{2m}}},f)\]
 for sufficiently large $r$.
\end{proof}
\noindent\textbf{Proof of Theorem 1:}
We will prove the result through the following steps:\\
Step 1: Using the given hypothesis, there exist real numbers $\alpha >1$ and  $r_0$ such that for each $r\geq r_0$, there exists $\sigma$ satisfying $r\leq \sigma \leq r^{\alpha}$ and $L(\sigma,f)=M(r,f)$ \cite{Singh}.\\
Step 2: In this step, we will observe some inequalities:
\begin{enumerate}
\item [(i)] Using \Cref{lemma1}, for $m=\alpha$ there exists $r_1$ such that 
\[\left(\log M(r^{\frac{1}{2\alpha}},f)^{\frac{1}{2\alpha}M(r^{\frac{1}{2\alpha}},f)^{\frac{1}{2\alpha}}}\right)^{\alpha}\leq M(\log r^{\frac{1}{2\alpha}r^{\frac{1}{2\alpha}}},f)
\text{ for }r\geq r_1.\] 
\item [(ii)] From \cite[Lemma 2.2]{rippon}, for $c=2\alpha$ there exists $r_2$ such that 
\[M(r^{2\alpha},f)\geq M(r,f)^{2\alpha} \text{ for }r\geq r_2.\]
\end{enumerate}
Step 3: Observe that the sequence $\{\log R_n^{\frac{1}{2\alpha}R_n^{\frac{1}{2\alpha}}}\}$ tends to infinity as $n$ tends to infinity.
Using Step 1, for each $n$ there exists $\sigma_n$ satisfying 
\[ \log R_n^{\frac{1}{2\alpha}R_n^{\frac{1}{2\alpha}}}\leq \sigma_n\leq \left(\log R_n^{\frac{1}{2\alpha}R_n^{\frac{1}{2\alpha}}}\right)^{\alpha}\]
such that $L(\sigma_n,f)=M(\log R_n^{\frac{1}{2\alpha}R_n^{\frac{1}{2\alpha}}},f)$. 
Also, by Step 2, we have
\begin{align*}
L(\sigma_n,f)&= M(\log R_n^{\frac{1}{2\alpha}R_n^{\frac{1}{2\alpha}}},f)\\
&\geq \left(\log M(R_n^{\frac{1}{2\alpha}},f)^{\frac{1}{2\alpha}M(R_n^{\frac{1}{2\alpha}},f)^{\frac{1}{2\alpha}}}\right)^{\alpha}\\
&=\left(\log R_{n+1}^{\frac{1}{2\alpha}R_{n+1}^{\frac{1}{2\alpha}}}\right)^{\alpha}.
\end{align*}
Hence $\liminf_{n\to\infty} L(\sigma_n,f)=\infty$. This gives us that the image of unbounded Fatou component is unbounded \cite{hinka}.\\
Step 4: Suppose that $F(f)$ has an unbounded component, say $U$. 
Without loss of generality, we can assume that $0,1 \in J(f)$ such that $f(0)=1$.
As $U$ is an unbounded component, for sufficiently large $n\geq n_2 (\geq n_1)$, the Fatou component $U$ intersects the following three circles:
\begin{align*}
T_n&=\{z: |z|=S_n\}, \\
T_n^1&=\{z:|z|=\sigma_n\}, \text{ and}\\
T_n^2&=\left\{z:|z|=\left(\log R_{n}^{\frac{1}{2\alpha}R_{n}^{\frac{1}{2\alpha}}}\right)^{\alpha}\right\}.
\end{align*} 
As $U$ is a connected Fatou component, we can join any two points by a path. 
Let $\gamma:[0,1]\to U$ be a path joining the points $z_n\in T_n$ in $U$ and $z_{n+1}^2\in T_{n+1}^2$ in $U$. 
Then $\gamma$ intersects the circle $T_{n+1}^1$ at some point, say $z_{n+1}^1$.
By Step 3, $f(U)$ is an unbounded Fatou component containing the path $f\circ\gamma$.
Also, we have $|f(z_n)|\leq S_{n+1}$ and \[|f(z_{n+1}^1)|\geq \left(\log R_{n+2}^{\frac{1}{2\alpha}R_{n+2}^{\frac{1}{2\alpha}}}\right)^{\alpha}.\] 
These observations imply the existence of two points $z_{n+1}\in T_{n+1}$ and $z_{n+2}^2\in T_{n+2}^2$ which lie on the the path $f\circ\gamma$.
\begin{figure}
\centering
\includegraphics[totalheight=8cm]{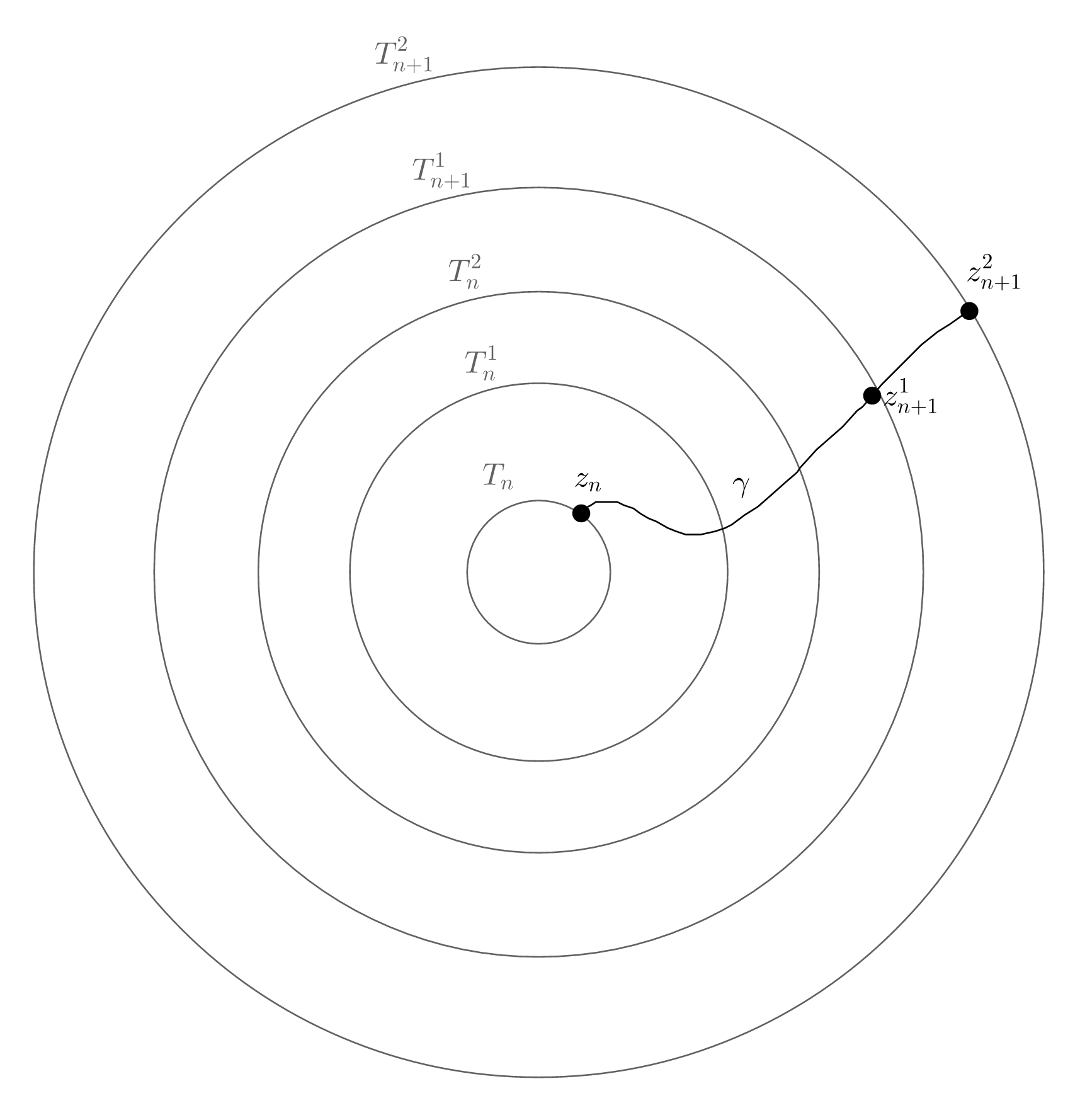}
\caption{ }
\label{fig:1}
\end{figure}

By continuing this process inductively, we get that $f^k(U)$ will contain the path $f^k\circ\gamma$ which intersects the circles $T_{n+k}$ at $z_{n+k}$ and $T_{n+k+1}^2$ at $z_{n+k+1}^2$.
Therefore, $f^k$ takes a value of modulus at least  $S_{n+k}$ on $\gamma$. 
This shows that $\{f^k\}$ goes to infinity locally uniformly on $U$. 
This, in particular, implies that there exists $N_0\in\mathbb{N}$ such that for any $k>N_0$ and for every $z\in\gamma([0,1])$, we have $|f^k(z)|>1$.
Applying \cite[Lemma 5]{baker1} to the compact set $\gamma([0,1])$, for all $z,w\in \gamma([0,1])$, we have
\[|f^k(z)|<B|f^k(w)|^C \text{ for every }k>N_0.\] 
Now, for any $k>N_0$, we can choose $u_k,u_k^2\in\gamma([0,1])$ such that $f^k(u_k)=z_{n+k}$ and $f^k(u_k^2)=z_{n+k+1}^2$. 
Hence, we have $|z_{n+k+1}^2|<|z_{n+k}|^s$ for any $k>N_0$, i.e.,

 \[\left(\log R_{n+k+1}^{\frac{1}{2\alpha}R_{n+k+1}^{\frac{1}{2\alpha}}}\right)^{\alpha}< B S_{n+k}^C \text{ for any } k>N_0.\]
 
Now, using the relation between the sequences $R_n$ and $S_n$, we have
 \[M(S_{n+k},f)=S_{n+k+1}< \left(\log R_{n+k+1}^{\frac{1}{2\alpha}R_{n+k+1}^{\frac{1}{2\alpha}}}\right)^{\alpha} < B S_{n+k}^C\text{ for every } k>N_0,\] 
  which is a contradiction as $f$ is a transcendental entire function.\qed\\
As a consequence, we have the following corollary giving a partial answer to Baker's question.
\begin{cor}
Suppose $f$ is a transcendental entire function of order $\rho <\frac{1}{2}$, minimal type. Then $F(f)$ has no unbounded component.
\end{cor}
\begin{proof}
As $\rho< \frac{1}{2}$, then by \cite{baker2}, the conclusion of the Step 1 of the proof of \Cref{THM1} is already satisfied. This gives us that $F(f)$ has no unbounded component.
\end{proof}
\section{Acknowledgement}
The research of the author is supported by the National Board for Higher Mathematics, India.

\end{document}